\documentclass[a4wide, reqno, 11pt]{amsart}
\usepackage[english]{babel}
\usepackage{amsmath}
\usepackage{amssymb}
\usepackage{enumerate}
\usepackage{ifthen}
\usepackage{hyperref}
\usepackage{color}
\newcommand{\margnote}[1]{
\ifthenelse{\boolean{shownotes}}%
{\marginpar{\raggedright\tiny\texttt{#1}}}%
{}%
}

\newcommand{\hole}[1]{
\ifthenelse{\boolean{shownotes}}%
{\begin{center} \fbox{ \rule {.25cm}{0cm}
\rule[-.1cm]{0cm}{.4cm} \parbox{.85\textwidth}{\begin{center}
\texttt{#1}\end{center}} \rule {.25cm}{0cm}}\end{center}}
{}
}
\newtheorem{theorem}{Theorem}[section]

\newtheorem{lemma}[theorem]{Lemma}

\newtheorem{definition}[theorem]{Definition}

\theoremstyle{remark}
\newtheorem{remark}[theorem]{Remark}

\newcommand{\e}{\varepsilon}		       
\newcommand{\R}{\mathbb{R}}

\newcommand{\D}{\mathrm{D}}
\newcommand{\ue}{u^{\varepsilon}}
\newcommand{\pe}{p^{\varepsilon}}

\newcommand{\g}{g^{\varepsilon}}
\newcommand{\dive}{\mathop{\mathrm {div}}}

\numberwithin{equation}{section}

\begin{document}

\title[Suitable solutions by artificial compressibility]{On the
  construction of suitable weak solutions to the 3D Navier-Stokes
  equations in a bounded domain by an artificial compressibility
  method}

\author[Berselli]{Luigi C. Berselli}
\address[L. C. Berselli]{\newline
Dipartimento di Matematica\\ I-56127, Pisa, Italy}
\email[]{\href{berselli@dma.unipi.it}{berselli@dma.unipi.it}}

\author[Spirito]{Stefano Spirito}
\address[S. Spirito]{\newline
GSSI - Gran Sasso Science Institute \\ 67100, L'Aquila, Italy}
\email[]{\href{stefano.spirito@gssi.infn.it}{stefano.spirito@gssi.infn.it}}



\begin{abstract}
  In this paper we will prove that suitable weak solutions of three
  dimensional Navier-Stokes equations in bounded domain can be
  constructed by a particular type of artificial compressibility
  approximation.
\end{abstract}

\maketitle
\section{Introduction}
\label{sec:Intro}
Let $\Omega\subset\R^3$ be a bounded domain with smooth boundary
$\Gamma=\partial\Omega$ and $T>0$ be a fixed real number. We consider
the three-dimensional Navier-Stokes equations with unit viscosity and
zero external force:
\begin{equation}
  \label{eq:nse}
\begin{aligned}
  \partial_t u-\Delta u+(u\cdot\nabla)\,u+\nabla p&=0\qquad\textrm{in
  }\Omega\times(0,T), 
  \\
  \nabla\cdot u&=0\qquad\textrm{in }\Omega\times(0,T),
\end{aligned}
\end{equation}                                                                                                   
and the Navier (slip without friction) boundary conditions for $u$,
namely
\begin{equation}
  \label{eq:bc}
  \begin{aligned}
    u\cdot n=0 \qquad\textrm{
    on }\Gamma\times(0,T), 
    \\
    n\cdot \D u\cdot\tau=0 \qquad\textrm{
    on }\Gamma\times(0,T), 
  \end{aligned}
\end{equation}
where $\D u$ stands for the symmetric part of $\nabla u$, $n$ is the
unit normal vector on $\Gamma$, while $\tau$ denotes any unit
tangential vector on $\Gamma$ (Recall also that for incompressible
fluids $\Delta u=2\dive \D u$, to compare with~\eqref{eq:app} in the
compressible approximation).

The system~\eqref{eq:nse} is coupled with a divergence-free and
tangential to the boundary initial
datum
\begin{equation}
  \label{eq:id}
  u(x,0)=u_0(x)\qquad\textrm{on  }\Omega\times\{t=0\}.
\end{equation}
For the initial value boundary problem~\eqref{eq:nse}-\eqref{eq:bc} it
is well-known that, for any tangential and divergence-free vector
field $u_0\in L^{2}(\Omega)$, there exists at least a global weak
solution in the sense of Leray-Hopf,
see~\cite{AR2014,Bei2006,Bei2007c,BMR2007,SS1973,XX2007} for the
generalization of the classical results of Hopf~\cite{Hop1951}, which
was obtained for the Dirichlet case.
\begin{remark}
  In the sequel it will be enough to assume that $\Gamma\in C^{1,1}$
  and, to avoid technical complications, we also assume that the
  domain is simply connected and that it cannot be generated by
  revolution around a given axis (this could be relevant in the steady
  case, when dealing with the symmetric deformation tensor and Korn
  type inequalities), see~\cite{Bei2004}.
\end{remark}
On the other hand, the uniqueness and the regularity of the weak
solutions represents a problem still open and very far to be
understood. The best regularity result which is available for weak
solutions of the three-dimensional Navier-Stokes equations is the
Caffarelli-Kohn-Nirenberg theorem~\cite{CKN1982}, which asserts that
the velocity is smooth out of a set of parabolic Hausdorff dimension
zero. For example it implies that there are not space-time curves of
singularity in the velocity. However, the Caffarelli-Kohn-Nirenberg
theorem holds only for a particular subclass of weak solutions, called
in literature ``\textit{suitable weak solutions},'' see
Section~\ref{sec:Pre} for the precise definitions. Roughly speaking a
suitable weak solution is a particular Leray-Hopf weak
solution which, in addition to the global energy inequality, satisfies
in the sense of distributions also the following entropy-type
inequality:
\begin{equation}
  \label{eq:gen}
  \partial_{t}\left(\frac{1}{2}|u|^{2}\right)+\nabla\cdot\left(\left(\frac{1}{2}|u|^{2}+
      p\right)u\right)- 
  \Delta\left(\frac{1}{2}|u|^{2}\right)+ 
  |\nabla u|^{2}
\leq 0. 
\end{equation}
Inequality~\eqref{eq:gen} is often called in literature {\em
  generalized energy inequality} and is the main tool to prove the
partial regularity theorems~\cite{CKN1982,LS1999}. Since uniqueness of
Leray-Hopf weak solutions (within the same class of solutions) is not
known, each method used to prove existence of weak solutions may
possibly lead to a different weak solution. This in turn implies that
it is a very interesting problem to understand which ones of the different
approximation methods (in particular those important in applications
and in the construction of numerical solutions), provide existence of
suitable weak solutions.  This question has been considered for many
approximation methods~\cite{Bei1985a,Bei1985b,Bei1986,BCI2007}, and it
is worth to point out that the solutions constructed by the Leray
method~\cite{Ler1934} and the Leray-$\alpha$ variant turn out to be
suitable. Particularly important are the results obtained by
Guermond~\cite{Gue2006,Gue2007} where it was proved that some special
Galerkin methods lead to suitable weak solutions. However,
understanding whether solutions obtained by general Galerkin methods
(especially those ones based on Fourier series expansion) are suitable
has been posed in~\cite{Bei1985b} and it is still not completely
solved.

In the spirit of understanding whether certain methods used in the
numerical approximation produce suitable weak solutions (problem
emphasized in~\cite{Gue2008}) in this note we prove that weak
solutions obtained as a limit of a particular \textit{artificial
  compressibility} method are so.  More precisely, the approximation
system we consider is the following:
\begin{equation}
  \label{eq:app}
  \begin{aligned}
    \partial_t\ue+(\ue\cdot\nabla)\,\ue+\frac{1}{2}\ue\dive\ue-2\,{\dive\D\ue}
    +\nabla\pe&=0\qquad\textrm{in }\Omega\times(0,T),
    \\
    -\e\,\Delta\pe+\dive\ue&=0\qquad\textrm{in }\Omega\times(0,T),
\end{aligned}
\end{equation}
is coupled with Navier boundary conditions
\begin{equation}
  \label{eq:bc1}
  \begin{aligned}
    \ue\cdot n=0 \qquad\textrm{
    on }\Gamma\times(0,T), 
    \\
    n\cdot \D \ue\cdot\tau=0 \qquad\textrm{
    on }\Gamma\times(0,T), 
  \end{aligned}
\end{equation}
and a tangential to the boundary initial datum
\begin{equation}
  \label{eq:id1}
  \ue(x,0)=u_0^{\e}(x)\qquad\textrm{on }\Omega\times\{t=0\}.
\end{equation}
Concerning the pressure, it is natural to impose on it Neumann
boundary conditions,
\begin{equation}
  \label{eq:btre}
  \frac{\partial\pe}{\partial n}=0\qquad\textrm{ on }\Gamma\times(0,T),
\end{equation}
and normalization by zero average over the domain:
\begin{equation}
  \label{eq:preav}
  \int_{\Omega}\pe(x,t)\,\,dx=0\qquad \text{a.e. }t\in(0,T). 
\end{equation}
\begin{remark}
  The nonlinear terms in~\eqref{eq:app} can be also be written in the
  following equivalent divergence-type form
  \begin{equation}
    \label{eq:def:nl}
    nl(\ue,\ue):=  (\ue\cdot\nabla)\,\ue+\frac{1}{2}\,\ue\dive\ue=
    \dive(\ue\otimes\ue)-\frac{1}{2}\,\ue\dive\ue,  
  \end{equation}
where as usual $(a\otimes b)_{ij}:=a_i b_j$, for $i,j=1,2,3$. The
factor $\frac12\,\ue\dive\ue$, which vanishes in the incompressible limit,
is needed to keep the usual energy estimates.
\end{remark}
We recall that artificial compressibility methods were first studied
by Chorin~\cite{Cho1968,Cho1969} and Temam~\cite{Tem1968} and are
relevant in numerical analysis since they relax the divergence-free
constraint, which has a high computational cost. See also the recent
results in~\cite{Rus2012} for the Cauchy problem with $-\e\,\Delta
\pe$ replaced by $\e\,\pe$ (as in~\cite{Bei2004}) in
equations~\eqref{eq:app}. In addition, the case where the term
$-\e\,\Delta\pe$ in equations~\eqref{eq:app} is replaced by
$\e\,\partial_t\pe$ has been considered by different authors,
see~\cite{Tem1969,DM2006}.  Moreover, the convergence when $\e\to 0$
to a suitable weak solution in the whole space was considered
in~\cite{DS2011} for the scheme with the time derivative of the
pressure.

Here, we focus especially on the fact that we have a domain with solid
boundaries and the coupling with the Navier boundary conditions makes
possible to obtain appropriate estimates on the pressure. The main
theorem of this note is the following. See Section~\ref{sec:Pre} for
the notations and the main definitions.
\begin{theorem}
  \label{teo:main}
  Let $\Omega$ be bounded and of class $C^{1,1}$.  Let
  $\{(\ue,\pe)\}_{\e}$ be a sequence of weak solutions of the initial
  value boundary problem~\eqref{eq:app}-\eqref{eq:btre} with
  \begin{equation*}
      \dive\ue_0=0\qquad\text{in }\Omega,
  \end{equation*}
  and $\{\ue_0\}_\e$ bounded uniformly in $H^1_\tau(\Omega)$ such that 
 \begin{equation*}
 \ue_0\to u_0\qquad\textrm{ strongly in }L^{2}(\Omega).
 \end{equation*}
 Then, the
  following $\e$-independent estimate on the pressure holds true
\begin{equation}
  \label{eq:m2}
\exists\, C>0:\qquad  \|\pe\|_{ L^{\frac{5}{3}}((0,T)\times\Omega)}\leq
C,\qquad \forall\,\e>0.
\end{equation}
Moreover, there exists $u$, a Leray-Hopf weak solution
of~\eqref{eq:nse}-\eqref{eq:bc}, and an associated pressure $p\in
L^{\frac{5}{3}}((0,T)\times\Omega)$ such that
\begin{align}
  &\ue\overset{*}{\rightharpoonup} u\qquad\textrm{ weakly* in
  }L^{\infty}(0,T;L^{2}(\Omega)),\label{eq:m3}
  \\
  &\nabla\ue\rightharpoonup \nabla u\qquad\textrm{ weakly in
  }L^{\infty}(0,T;L^{2}(\Omega)),
\label{eq:m4}
  \\
  &\pe\rightharpoonup p\qquad\textrm{ weakly in
  }L^{\frac{5}{3}}((0,T)\times\Omega).
\label{eq:m5}
\end{align}
Finally, $(u,p)$ is a
suitable weak solution to the Navier-Stokes
equations~\eqref{eq:nse}-\eqref{eq:bc}-\eqref{eq:id}.
\end{theorem}
The existence of weak solutions of~\eqref{eq:app}-\eqref{eq:btre} is
quite standard and can be easily proved by smoothing suitably the non
linear terms.  The main obstacle to prove Theorem~\ref{teo:main} is to
get $\e$-independent estimates on the pressure. Since we do not have
the divergence-free constraint we cannot deduce independent estimates
on the pressure by using the classical elliptic equations associated
to the pressure. Moreover, also methods based on the semigroup theory
as in~\cite{SvW1986} seem not directly working here, since the
approximation system does not immediately fit in that abstract
framework. We will get the necessary a priori estimates on the
pressure from the momentum equations as in the case of compressible
Navier-Stokes equations~\cite{FP2000b,Lio1998}.  This approach has
been introduced in~\cite{BMR2007} to study a class of non-Newtonian
fluids and used also in~\cite{CL2014} to address the analysis of
models for the evolution of the turbulent kinetic energy, but without
considering the question of the local energy inequality.  In
particular, we point out that it is in the pressure estimate that we
need to employ the Navier boundary conditions, since they allow us to
control the various term arising in the integration by parts.
\subsection*{Plan of the paper}
The plan of the paper is the following. In Section~\ref{sec:Pre} we
recall the main definitions and the main tools we will use in the
proof of Theorem~\ref{teo:main}. In Section~\ref{sec:ape} we prove the
main a priori estimates needed in the proof of Theorem~\ref{teo:main}
and finally in Section~\ref{sec:proof} we prove
Theorem~\ref{teo:main}.
\section{Preliminaries}
\label{sec:Pre}
In this section we fix the notations and we recall the main
definitions and tools we will use to prove Theorem~\ref{teo:main}.
Given $\Omega\subset\R^3$, the space of $C^{\infty}$ functions or
vector fields on $\Omega$ tangential to the boundary will be denoted
by $C^{\infty}_{\tau}(\Omega)$.  We will denote with $L^{p}(\Omega)$
the standard Lebesgue spaces and with $\|\cdot\|_p$ their norm. The
classical Sobolev space is denoted by $W^{1,2}(\Omega)$ and its norm
by $\|\cdot\|_{k,p}$ and when $k=1$ and $p=2$ we denote
$W^{k,p}(\Omega)$ with $H^{1}(\Omega)$. As usual we denote the
$L^2$-scalar product by $(\,.\,,\,.\,)$. Finally, we denote by
$L^{2}_{\sigma,\tau}(\Omega)$ and $H^{1}_{\sigma,\tau}(\Omega)$ the
space of the tangential to the boundary and divergence-free vector
fields respectively in $L^{2}(\Omega)$ and $H^{1}(\Omega)$. By
``$\,\cdot\,$'' we denote the scalar product between vectors, while by
``$\,:\,$'' we denote the complete contraction of second order
tensors. Finally, given a Banach space we denote by $C_w([0,T]; X)$
the space of continuous function from the interval $[0,T]$ to the
space $X$ endowed with the weak topology.

Let us start by giving the precise definition of Leray-Hopf weak
solution for the initial value boundary
problem~\eqref{eq:nse}-\eqref{eq:bc}.
\begin{definition}
  \label{def:lhws}
  A vector field $u$ is a Leray-Hopf weak solution to the
  Navier-Stokes equations~\eqref{eq:nse}-\eqref{eq:bc}-\eqref{eq:id}
  if the following proprieties hold true:
  \\
  1) The velocity $u$ satisfies $u\in
  C_w(0,T;L^{2}_{\sigma,\tau}(\Omega))\cap L^{2}(0,T;H_{\sigma,
    \tau}^{1}(\Omega))$;
  \\
  2) The velocity $u$ satisfies the following integral identity
\begin{equation*}
  \begin{aligned}
    \int_{0}^{T}\int_{\Omega}\Big(u\cdot \partial_t\phi-\nabla
    u:\nabla\phi-(u\cdot\nabla)\,u\cdot\phi\Big)\,dx ds{-}\int_0^t\int_{\Gamma}u\cdot
    \nabla n\cdot \phi\,dS ds
    \\
    =-\int_{\Omega}u_0\cdot \phi(0)\,dx,
  \end{aligned}
\end{equation*}
for all smooth vector fields $\phi$, divergence-free,  tangential to the
  boundary, and such that $\phi(T)=0$;
\\
3)  The velocity $u$ satisfies the global energy inequality 
\begin{equation}
  \label{eq:ei}
\frac{1}{2}  \|u(t)\|_{2}^2+\int_0^t\|\nabla
  u(s)\|_{2}^2\,ds-\int_0^t\int_{\Gamma}u\cdot\nabla n \cdot
  u\,dS ds\leq \frac{1}{2}\|u_0\|_{2}^2. 
\end{equation}
\end{definition}
\begin{remark}
In the above definition and in the sequel, we recall that 
\begin{equation*}
\int_{\Gamma}u\cdot\nabla n \cdot
  u\,dS:=  \int_{\Gamma}\sum_{i,j=1}^3 u_i u_j\frac{\partial 
      n_i}{\partial x_j}\,dS.
\end{equation*}
and we also recall that the boundary condition~\eqref{eq:bc} are
encoded in the weak formulation, see also~\cite{IP2006}.
\end{remark}
The next definition we recall is that of suitable weak solutions.
\begin{definition}
\label{def:suitable-NSE}
A pair $(u,p)$ is a suitable weak solution to the Navier-Stokes
equations~\eqref{eq:nse}-\eqref{eq:bc}-\eqref{eq:id} if the following
properties hold true:
\\
1) $u$ is a Leray-Hopf weak solution with associated a pressure $p$ such that 
\begin{equation}
p\in L^{\frac{5}{3}}((0,T)\times\Omega);
\end{equation}
\\
2) The generalized energy inequality holds true
  \begin{equation}
    \label{eq:generalized_energy_inequality}
    \begin{aligned}
      \int_0^T\int_\Omega|\nabla u|^2\phi\,dx ds
      &\leq\int_{0}^{T}\int_\Omega\frac{|\ue|^{2}}{2}\big(\phi_{t}+\Delta\phi\big)
      \,dx  ds
      \\
      &\qquad +\Big(\ue\frac{|\ue|^{2}}{2}+\pe\ue\Big)\cdot\nabla\phi\,dx ds,
\end{aligned}
\end{equation}
    for all non-negative $\phi\in C^{\infty}_{c}((0,T)\times\Omega)$.
\end{definition}
We want to point out that the real difference between the Leray-Hopf
weak solutions and the suitable weak solutions relies in the local
energy inequality. Indeed, it is always possible to associate to a
Leray-Hopf weak solutions $u$ a pressure $p$ (modulo arbitrary
functions of time) which belongs to the space
$L^{\frac{5}{3}}((0,T)\times\Omega)$. This is standard in a setting
without physical boundaries by using the Stokes system. In the case of
Dirichlet boundary conditions this was first proved by Sohr and von
Wahl in~\cite{SvW1986}. Finally, in the case of Navier boundary
conditions this can be deduced again by the elliptic equations
associated to the pressure $p$, see~\cite{Ber2010b}.  On the other
hand, it is not known how the deduce the generalized energy
inequality~\eqref{eq:gen} since the regularity of a Leray-Hopf weak
solution is not enough to justify the chain rule in the time
derivative term and in the non linear term.

We pass now to the precise definitions concerning the compressible
approximation we will consider.
\begin{definition}
  \label{def:weak-solution-compressible}
  The couple $(\ue,\pe)$ is a weak solution to the compressible
  approximation~\eqref{eq:app}-\eqref{eq:bc1}-\eqref{eq:id1}-\eqref{eq:btre}-\eqref{eq:preav} 
  if:
\\ 
1) The velocity and pressure $(\ue,\pe)$ satisfy
    \begin{equation*}
      \begin{aligned}
        &\ue\in L^{\infty}(0,T;L^{2}(\Omega))\cap
        L^{2}(0,T;H^{1}_\tau(\Omega)),
        \\
        &\sqrt{\e}\,\nabla\pe\in L^{2}((0,T)\times\Omega);
      \end{aligned}
    \end{equation*}
\\
2) The velocity and pressure $(\ue,\pe)$ satisfy the following integral identity
  \begin{equation*}
    \begin{aligned}
      \int_0^T\int_\Omega\Big[-\ue\cdot\partial_t\phi+nl(\ue,\ue)\cdot\phi+
      2\D\ue:\D\phi+\nabla\pe\cdot\phi
      \\
      +\e\,\nabla\pe\cdot\nabla\psi+\dive\ue\,\psi\Big]\,dx ds
      =\int_\Omega\ue_0\cdot\phi(0)\,dx, 
    \end{aligned}
  \end{equation*}
  for all smooth vector fields $\phi$ tangential to the
  boundary, such that $\phi(T)=0$ and for all smooth scalar fields
  $\psi$ with zero mean value;
\\
3) The velocity and pressure $(\ue,\pe)$ satisfy the global energy inequality
\begin{equation}
      \label{eq:eiapp0}
      \begin{aligned} 
        \frac{1}{2}\|\ue(t)\|^2&
        +\int_{0}^{t}\Big[2\,\|\D\ue\|^2
        +\e\,\|\nabla\pe\|^2\Big]\,ds
        \leq\frac{1}{2}\|\ue_0\|^2\qquad \forall\,t\geq0;
      \end{aligned}
    \end{equation}
4) The velocity and pressure $(\ue,\pe)$ satisfy the local energy inequality
      \begin{equation}
        \label{eq:genapp0}
        \begin{aligned}
          &\int_{0}^{T}\int_\Omega\Big(2|\D\ue|^{2}+\e|\nabla\pe|^2\Big)\phi\,dx
          ds
          \leq\int_{0}^{T}\int_\Omega\frac{|\ue|^{2}}{2}\big(\phi_{t}+\Delta\phi\big)\,dx ds
          \\
          &+\int_0^{T}\int_\Omega\Big(\ue\frac{|\ue|^{2}}{2}+\pe\ue-\e\,\pe\nabla\pe+\ue\,\dive\ue\Big)\cdot
          \nabla\phi\,dx ds
          \\
          &\qquad+\int_0^{T}\int_\Omega u\otimes u:\nabla\phi\,dx ds,
        \end{aligned}
      \end{equation}
    for all non-negative $\phi\in C^{\infty}_{c}((0,T)\times\Omega)$.
\end{definition}
\begin{remark}
  The global and local energy estimates can be rewritten in a more
  useful form by performing some integration by parts. Observe that
  for any $v\in H^1_\tau(\Omega)$ it holds (cf.~\cite[Sec.~2]{Bei2004})
  that
  \begin{equation*}
    2\int_\Omega\D v:\D v\,dx=\|\nabla v\|^2+\|\dive v\|^2-
    \int_{\Gamma}v\cdot\nabla n \cdot
    v\,dS.
  \end{equation*}
  Moreover, by a standard compactness argument
  (see~\cite[Lem.~2.3]{Bei2004}), it follows that there exists
  $c=c(\Omega)$ such that
  \begin{equation*}
    2\int_\Omega \D v:\D v\,dx\geq c\|\nabla v\|^2\qquad \forall\, v\in H^1_\tau(\Omega), 
  \end{equation*}
  and the energy inequality~\eqref{eq:eiapp0} can be rewritten either as follows
    \begin{equation}
      \label{eq:eiapp}
      \begin{aligned} 
        \frac{1}{2}\|\ue(t)\|^2
        +\int_{0}^{t}\Big[ \|\nabla\ue\|^2&+\|\dive\ue\|^2
        +\e\|\nabla\pe\|^2
        \\
        &-\int_{\Gamma}\ue\cdot\nabla
        n\cdot \ue\,dS\Big] \,ds
        \leq\frac{1}{2}\|\ue_0\|^2,
      \end{aligned}
    \end{equation}
or even as 
    \begin{equation}
      \label{eq:eiapp1}
      \begin{aligned} 
        \frac{1}{2}\|\ue(t)\|^2
        +\int_{0}^{t}\Big[ c\|\nabla\ue\|^2+ \e\|\nabla\pe\|^2\Big]\,ds
        \leq\frac{1}{2}\|\ue_0\|^2.
      \end{aligned}
    \end{equation}
Concerning the local energy inequality~\eqref{eq:genapp0}, with several integration by
parts we get the following equivalent formulation (which is more
similar to the standard one for the incompressible case), 
      \begin{equation}
        \label{eq:genapp}
        \begin{aligned}
          &\int_{0}^{T}\int_\Omega\Big(|\nabla\ue|^{2}+|\dive\ue|^2+\e|\nabla\pe|^2\Big)\phi\,dx ds
          \leq\int_{0}^{T}\int_\Omega\frac{|\ue|^{2}}{2}(\phi_{t}+\Delta\phi)\,dx ds
          \\
          &+\int_0^{T}\int_\Omega\Big(\ue\frac{|\ue|^{2}}{2}+\pe\ue-\e\,\pe\nabla\pe-\ue
          \,\dive\ue\Big)\cdot \nabla\phi\,dx ds,
        \end{aligned}
      \end{equation}
      for all non-negative $\phi\in
      C^{\infty}_{c}((0,T)\times\Omega)$.
\end{remark}
The results in~\cite{BMR2007}, when specialized to the case $r=2$
imply the following theorem and the local energy inequality follows
from the improved regularity of the pressure $\nabla \pe\in
L^2((0,T)\times\Omega)$, which is valid (but not uniform) for all
$\e>0$.
\begin{theorem}
  Let be given $\ue_0\in L^2_\tau(\Omega)$, then there exists a weak
  solution to the compressible
  approximation~\eqref{eq:app}-\eqref{eq:bc1}-\eqref{eq:id1}-\eqref{eq:btre}-\eqref{eq:preav}. 
\end{theorem}
Finally, we recall the classical Aubin-Lions Lemma which will be
useful to obtain compactness in time in the passage to the limit
from~\eqref{eq:app} to~\eqref{eq:nse}.
\begin{lemma}
\label{lem:al}
Let $X$, $B$, and $Y$ be reflexive Banach spaces. For $\e>0$
let $\{\ue\}_\e$ be a family of functions uniformly bounded in
$L^{p}(0,T;X)$ with $p\geq 1$ and let $\{\partial_t\ue\}_{\e}$ be
uniformly bounded in $L^{r}(0,T;Y)$. with $r>1$.  If $X$ is
compactly embedded in $B$ and $B$ is continuously embedded in $Y$,
then $\{\ue\}_{\e}$ is relatively compact in $L^{p}(0,T;B)$.
\end{lemma}
\section{Some estimates on the pressure independent of $\e$}
\label{sec:ape}
In this section we prove the $\e$-independent estimate~\eqref{eq:m2}
for the pressure, which represents the most relevant technical part to
show the convergence towards a suitable weak solution. 
\begin{lemma}
\label{lem:est3}
Let $(\ue, \pe)$ be a weak solution of
system~\eqref{eq:app}-\eqref{eq:id1}-\eqref{eq:btre}-\eqref{eq:preav}. Then,
there exists a constant $C>0$, independent of $\e$, such that
\begin{equation*}
  \sup_{0< t< T}\e\,\|\pe(t)\|_{\frac{5}{3}}^{\frac{5}{3}}+
  \int_0^T\|\pe(s)\|_{\frac{5}{3}}^{\frac{5}{3}}\,ds \leq C. 
\end{equation*}
\end{lemma}
\begin{proof}
  Let $\alpha=\frac{5}{3}$ and let $\g$ be the unique solution
  (normalized by a vanishing mean value) of the
  following Poisson problem with Neumann boundary conditions
\begin{equation}
  \label{eq:e31}
  \left\{
    \begin{aligned}
    \Delta
    \g&=|\pe|^{\alpha-2}\pe-\int_\Omega|\pe|^{\alpha-2}\pe\,dx\quad\textrm{
      in }\Omega\times\{t\},
    \\
    \frac{\partial \g}{\partial n}&=0\quad\textrm{ on }\Gamma\times\{t\}.
\end{aligned}             
\right. 
\end{equation}
\begin{remark}
  The number $\alpha-2$ is negative but the expression we write is
  legitimate since we are not really dividing by zero in sets of
  positive measure. In fact, by using the regularity of $\pe$ from
  Definition~\ref{def:weak-solution-compressible} it turns out that
  the function $|\pe|^{\frac{5}{3}-2}\pe$ is well-defined and belongs
  at least to $L^3((0,T)\times\Omega)$, even if not uniformly in
  $\e>0$.
\end{remark}
We use now the vector field $\nabla\g$ as test function in the first
equation of the system~\eqref{eq:app}. Note that, $\nabla\g$ is
tangential to the boundary and for each fixed $\e>0$ and we have that
$\nabla\g$ is smooth enough to make the integral below
well-defined. Let $t\in(0,T)$, then we get (recall~\eqref{eq:def:nl})
\begin{equation*}
  \int_0^t\int_{\Omega}\big(\partial_t\ue+nl(\ue,\ue)-2\dive\D\ue+
\nabla\pe\big)\cdot\nabla\g\,dx ds=0. 
\end{equation*}
Integrating by parts the term involving the deformation gradient by
using the Navier conditions~\eqref{eq:id1} and the one involving the
pressure and by using the definition of $\g$, the Neumann boundary
conditions, and the fact that $\pe$ has zero average we get
\begin{equation}
  \label{eq:l2}
  \int_0^t\int_\Omega|\pe|^{\frac{5}{3}}\,dx ds=\int_0^t\int_\Omega\Big[\big(\partial_t\ue
  +nl(\ue,\ue)\big)\cdot\nabla\g+2\,\D\ue:\nabla^2\g\Big]\,dx ds.  
\end{equation}
Let us consider the term with the time derivative. We can use the
Helmholtz decomposition to write
\begin{equation*}
  \ue=P\ue+\nabla q.
\end{equation*}
where $P$ is the Leray projector. Note that, since $\ue\cdot n=0$ on
$\Gamma\times(0,T)$ it holds that $n\cdot \nabla q=0$ on
$\Gamma\times(0,T)$ as well.  By using the second equation
of~\eqref{eq:app} we have that %
$ \Delta \pe=\frac{1}{\e}\dive(P\ue+\nabla q)=\frac{1}{\e}\dive\nabla
q=\frac{1}{\e}\Delta q$.
This means that for a.e. $t\in(0,T)$ we have in a weak sense
\begin{equation*}
\left\{
  \begin{aligned}
    \Delta
    \left(\pe-\frac{q}{\e}\right)&=0\qquad\textrm{in}\qquad\Omega\times\{t\},
    \\
    \frac{\partial}{\partial
      n}\left(\pe-\frac{q}{\e}\right)&=0\qquad\textrm{in}\qquad\Gamma\times\{t\}.
\end{aligned}
\right.
\end{equation*}
Then, by the uniqueness of the Neumann problem under the usual
normalization of zero mean value, we get that
\begin{equation*}
\pe=\frac{q}{\e}.
\end{equation*}
Then, we observe that since $\ue\cdot n=0$ on $\Gamma$ and $\dive
P\ue=0$ in $\Omega$, then
\begin{equation*}
  \begin{aligned}
    \int_0^t\int_{\Omega}\partial_t\ue\cdot\nabla\g\,dx
    ds&=-\int_0^t\int_{\Omega}(\dive\partial_t\ue)\,\g\,dx
    ds=-\int_0^t\int_\Omega\Delta \partial_t q\,\g\,dx ds
    \\
    &=-\e\int_0^t\int_\Omega\Delta\partial_t\pe\,\g\,dx ds.
  \end{aligned}
\end{equation*}
We then integrate by parts twice and by observing that $n\cdot\nabla\pe=0$,
(hence a fortiori also $n\cdot\nabla\partial_t\pe=0$) and $n\cdot\nabla\g=0$
we obtain 
\begin{equation*}
  \begin{aligned}
        \int_0^t\int_{\Omega}\partial_t\ue\cdot\nabla\g\,dx
    ds=-\e\int_0^t\int_\Omega\partial_t\pe\,\Delta\g\,dx ds.
  \end{aligned}
\end{equation*}
Hence, by recalling the definition of $\g$ via the boundary value
problem~\eqref{eq:e31} we get 
\begin{equation*}
  \begin{aligned}
        \int_0^t\int_{\Omega}\partial_t\ue\cdot\nabla\g\,dx ds&=
        -\e\int_0^t\int_\Omega\partial_t\pe(x,s)\
        \pe(x,s)|\pe(x,s)|^{\alpha-2} \,dx ds
    \\
    &=-\frac{\e}{\alpha}\int_0^t\left(\frac{d}{ds}\int_{\Omega}|\pe(x,s)|^{\alpha}\,dx\right)\,ds,
  \end{aligned}
\end{equation*}
where we used again the fact that $\pe$ has zero average for any
$t\in(0,T)$.

Then, by recalling that $\alpha=\frac{5}{3}$, by integrating in
time the last term, and taking into account the fact that all the
initial data $\{\ue_0\}_\e$ are divergence-free (hence $\pe(0)=0$) we
get that the equation~\eqref{eq:l2} reads now
\begin{equation*}
  \begin{aligned}
    \frac{3\e}{5}\|\pe(t)\|^{\frac{5}{3}}_{\frac{5}{3}}+\int_{0}^t\|\pe(s)\|_{\frac{5}{3}}^{\frac{5}{3}}\,ds&=
    \int_0^t\int_{\Omega}nl(\ue,\ue)\cdot\nabla\g\,dx ds
    \\
   &\hspace{3cm}+ 
    2\int_0^t\int_{\Omega}\D\ue:\nabla^2\g\,dx ds
    \\
    &=:I_1+I_2.
  \end{aligned}
\end{equation*}
We estimate the first integral from the right-hand side as follows: 
\begin{equation*}
\begin{aligned}
  I_1&\leq \int_0^t\|nl(\ue,\ue)\|_{\frac{15}{14}}\|\nabla\g\|_{15}\,ds
  \\
  &\leq
  \int_{0}^{t}\|nl(\ue,\ue)\|_{\frac{15}{14}}\|D^2\g\|_{\frac{5}{2}}\,ds
  \\
  &\leq
  \int_{0}^{t}\|nl(\ue,\ue)\|_{\frac{15}{14}}\|\pe\|_{\frac{5}{3}}^{\frac{2}{3}}\,ds
  \\
  &\leq C
  \int_{0}^{t}\|nl(\ue,\ue)\|_{\frac{15}{14}}^{\frac{5}{3}}\,ds+\frac{1}{8}\int_0^t\|\pe\|_{\frac{5}{3}}^{\frac{5}{3}}\,ds, 
\end{aligned}     
\end{equation*}
where we have used H\"older, Sobolev, and Young inequalities. Finally,
by convex interpolation we have that if $\ue\in
L^\infty(0,T;L^2(\Omega))\cap L^2(0,T;H^1(\Omega))$, then the
nonlinear term $nl(\ue,\ue)$ belongs to
$L^{\frac{5}{3}}(0,T;L^{\frac{15}{14}}(\Omega))$.

Hence, we have proved that there exists a constant $C_1$, depending
only on the $L^2$-norm of the initial data (through the energy
inequality~\eqref{eq:eiapp1}) such that
\begin{equation*}
  I_1\leq C+\frac{1}{4}\int_0^t\|\pe\|_{\frac{5}{3}}^{\frac{5}{3}}\,ds.
\end{equation*}
Concerning the term $I_2$ we have, by using H\"older inequality and
again the elliptic estimates   
\begin{equation*}
\begin{aligned}
  \int_0^t\int_\Omega|\nabla\ue||\nabla^2\g|\,ds&
  \leq \int_0^t\|\nabla\ue\|_{\frac{5}{3}}\|D^2\g\|_{\frac{5}{2}}\,ds
  \\
  &\leq
  C(\Omega)\int_0^t\|\nabla\ue\|_2\|\pe\|_{\frac{5}{3}}^{\frac{2}{3}}\,ds
  \\
  &\leq C(\Omega,
  T)\left(\int_0^t\|\nabla\ue\|_2^2\right)^{\frac{6}{5}}\,ds
  +\frac{1}{4}\int_0^t\|\pe\|_{\frac{5}{3}}^\frac{5}{3}\,ds
  \\
  &\leq C_2+\frac{1}{4}\int_0^t\|\pe\|_{\frac{5}{3}}^\frac{5}{3}\,ds,
\end{aligned}
\end{equation*}
for a constant $C_2$, depending only on $\Omega,\,T$, and the
$L^2$-norm of the initial data (again through the energy
inequality~\eqref{eq:eiapp1}).  Collecting all estimates, we finally
get
\begin{equation*}
  \frac{3\e}{5}\|\pe(t)\|_{\frac{5}{3}}^{\frac{5}{3}}
  +\frac{1}{2}\int_{0}^{t}\|\pe(s)\|_{\frac{5}{3}}^{\frac{5}{3}}\,ds\leq
  C_1+C_2,
\end{equation*}
ending the proof. By the way we also proved that
\begin{equation*}
  \e^{\frac{3}{5}}\pe\text{ is bounded uniformly in }L^\infty(0,T;L^{\frac{5}{3}}(\Omega)),
\end{equation*}
even if this information will be not used in the sequel.
\end{proof}
\section{Proof of the Main Theorem}\label{sec:proof}
By using the estimates from the previous section we can now prove
Theorem~\ref{teo:main} in an elementary way by using standard weak
compactness method.  By recalling~\eqref{eq:eiapp}-\eqref{eq:m2}
we have that, up to sub-sequences, there exist $u\in
L^{\infty}(0,T;L^{2}(\Omega))\cap L^{2}(0,T;H^{1}(\Omega))$ and $p\in
L^{\frac{5}{3}}((0,T)\times\Omega)$ such that
\begin{equation}
  \label{eq:basicconv}
  \begin{aligned}
    &\nabla\ue\rightharpoonup\nabla u\qquad\textrm{ weakly in
    }L^{2}(0,T\times\Omega),
    \\
    &\ue\overset{*}{\rightharpoonup} u\qquad\textrm{ weakly* in
    }L^{\infty}(0,T;L^2(\Omega)),
    \\
    &\pe\rightharpoonup p\qquad\textrm{ weakly in
    }L^{\frac{5}{3}}((0,T)\times\Omega). 
  \end{aligned}
\end{equation}
Moreover, by~\eqref{eq:eiapp} if follows that $\sqrt{\e}\,\nabla \pe$ is
uniformly bounded in $L^{2}((0,T)\times\Omega)$, and then we have that
${\e}^{\frac{1}{2}+\delta}\,\nabla \pe$ converges strongly to zero for all
positive $\delta$. Hence,  in particular, we have
\begin{equation}
  \label{eq:epconv}
  \e\,\nabla\pe\rightarrow 0\qquad\textrm{ strongly in
  }L^{2}((0,T)\times\Omega). 
\end{equation}
It also follows that $u$ is divergence-free. Indeed, by the second
equation of~\eqref{eq:app} we get
\begin{equation*}
\begin{aligned}
  \int_{0}^{T}\int_\Omega\dive\ue\,\psi\,dx
  ds&=\sqrt{\e}\int_{0}^{T}<\sqrt{\e}\Delta\pe,\psi>_{H^{-1},H^1_0}\,ds
  \\
  &=
  \sqrt{\e}\int_{0}^{T}\int_\Omega\sqrt{\e}\nabla\pe\cdot\nabla\psi\,dx
  ds
  \\
  &\leq
  \sqrt{\e}\int_{0}^{T}\|\sqrt{\e}\nabla\pe\|_{2}\|\nabla\psi\|_2\,ds
  \\
  &\leq
  \sqrt{\e}\|\sqrt{\e}\nabla\pe\|_{L^2(0,T;L^2(\Omega))}
  \|\nabla\psi\|_{L^2(0,T;L^2(\Omega))}.
\end{aligned} 
\end{equation*}
By taking the supremum over the functions $\psi\in
L^2(0,T;H^1_0(\Omega))$ and by using the estimate
on $\pe$ from~\eqref{eq:eiapp} we get that
\begin{equation*}
  \dive\ue\rightharpoonup 0\qquad\textrm{ weakly in }L^{2}(0,T;H^{-1}(\Omega)). 
\end{equation*}
Then, by the uniqueness of weak limits we get that $u$ is
divergence-free.  

The next step is to prove the strong convergence of $\ue$ in
$L^{2}((0,T)\times\Omega)$.  We can now use Aubin-Lions,
Lemma~\ref{lem:al}, provided that we can show some estimates on the
time derivative of the velocity, and these are usual obtained by
comparison.  We observe that $\pe$ is uniformly bounded in
$L^{\frac{5}{3}}((0,T)\times\Omega))$, the nonlinear term
$nl(\ue,\ue)$ is uniformly bounded in
$L^{\frac{4}{3}}(0,T;H^{-1}(\Omega))$, and $\dive\D \ue$ is uniformly
bounded in $L^{2}(0,T;H^{-1}(\Omega))$, hence we get that
\begin{equation*}
  \partial_t \ue\in L^{\frac{4}{3}}(0,T;W^{-1,\frac{5}{3}}(\Omega)), \text{ uniformly
    with respect to }\e.
\end{equation*}
Then, by applying Lemma~\ref{lem:al} with $X=W^{1,2}(\Omega)$,
$B=L^2(\Omega)$, and $Y=W^{-1,\frac{5}{3}}(\Omega)$ we get
\begin{equation}
  \label{eq:strong2}
  \ue\rightarrow u\qquad\textrm{ strongly in
  }L^{2}((0,T)\times\Omega).
\end{equation}
Then, we investigate the convergence of $P\ue$ and $Q\ue$, where $P$
is the Leray projector and $Q:=I-P$.  By applying $P$ to the first
equation of~\eqref{eq:app} we get that
\begin{equation*}
  \partial_t P\ue-P(\Delta
  \ue)+P((\ue\cdot\nabla)\,\ue)-P\left(\frac{1}{2}\,(\ue\dive\ue)\right)=0. 
\end{equation*}
From this equation we show, again by comparison, that
$\partial_{t}P\ue$ is bounded in
$L^{\frac{4}{3}}(0,T;(H^1_{\sigma,\tau}(\Omega))')$ uniformly with respect to
$\e$, and since $P\ue$ is bounded in
$L^{2}(0,T;H^{1}_{\sigma,\tau}(\Omega))$ we can apply again
Lemma~\ref{lem:al} to obtain in a very standard way that
\begin{equation*}
  P\ue\rightarrow Pu=u\quad \textrm{ strongly in }L^{2}((0,T)\times\Omega)).
\end{equation*}
Then, we have that 
\begin{equation*}
\begin{aligned}
  \|Q\ue\|_{2}&=\|\ue-P\ue\|_2=\|\ue-u+u-P\ue\|_2
  \\
                  &\leq\|\ue-u\|_2+\|Pu-P\ue\|_{2}\to 0,\qquad\text{as }\e\to0.
\end{aligned}
\end{equation*}
This in turn implies,
\begin{equation*}
  Q\ue\rightarrow 0\textrm{ strongly in }L^{2}((0,T)\times\Omega)).
\end{equation*}
The next step is to prove that $Q\ue$ converge to $0$ strongly in
$L^{\frac{5}{2}}((0,T)\times\Omega))$. By using an interpolation
inequality and a Sobolev embedding theorem together with the
Poincar\'e inequality (valid for tangential vector fields,
see~\cite{Gal2011}) we get
\begin{equation*}
  \|Q\ue\|_{\frac{5}{2}}^{\frac{5}{2}}\leq \|Q\ue\|_{2}^{\frac{7}{4}}\|Q\ue\|_{6}^{\frac{3}{4}}
  \leq C\|Q\ue\|_{2}^{\frac{7}{4}}\|\nabla Q\ue\|_{2}^{\frac{3}{4}},
\end{equation*}
and due to the definition of $Q\ue=\ue-P\ue=\nabla q$. By integrating
in time we get
\begin{equation*}
\begin{aligned}
  \int_{0}^{T}&\|Q\ue(s)\|_{\frac{5}{2}}^{\frac{5}{2}}\,ds
  \leq
  c\int_{0}^{T}\|Q\ue(s)\|_{2}^{\frac{7}{4}}\|\nabla\ue(s)\|_{2}^{\frac{3}{4}}\,ds
  \\
  %
  %
  &\leq\sup_{0<t< T}\|Q\ue(t)\|_{2}^{\frac{1}{2}}\left(\int_{0}^{T}\|Q\ue(s)\|_{2}^{2}\,ds\right)^{\frac{5}{8}}
  \left(\int_{0}^{T}\|\nabla\ue(s)\|_{2}^{2}\,ds\right)^{\frac{3}{8}}. 
\end{aligned}
\end{equation*}
Then, by using the regularity of $(\ue,\pe)$ and  construction of the
Helmholtz decomposition, it holds that $Q\ue\in
L^\infty(0,T;L^2(\Omega))$ and we get 
\begin{equation}
  \label{eq:convQ}
  Q\ue\rightarrow 0\qquad\textrm{ strongly in
  }L^{\frac{5}{2}}((0,T)\times\Omega)). 
\end{equation}
The last convergence we need to prove concerns the term
$\e\,\nabla\pe$. By considering the second equation of~\eqref{eq:app} we
get that
\begin{equation*}
\e\,\Delta\pe=\dive\ue=\dive Q\ue.
\end{equation*}
By standard using estimates on the elliptic equations with Neumann
boundary conditions, see~\cite{GT1998}, we get the following estimate:
\begin{equation*}
  \e\,\|\nabla\pe\|_{\frac{5}{2}}\leq \|Q\ue\|_{\frac{5}{2}}.
\end{equation*}
By integrating in time and using~\eqref{eq:convQ} we get  then
\begin{equation*}
  \e\left(\int_0^t\|\nabla\pe\|^{\frac{5}{2}}_{\frac{5}{2}}\right)^{\frac{2}{5}}\rightarrow 0.
\end{equation*}
By using~\eqref{eq:basicconv},~\eqref{eq:epconv},~\eqref{eq:strong2},
and~\eqref{eq:eiapp} it is straightforward to prove that $u$ is a
Leray-Hopf weak solution. In order to prove that $(u,p)$ is a suitable
weak solution of the Navier-Stokes
equations~\eqref{eq:nse}-\eqref{eq:bc} we only have to show is that
$(u,p)$ satisfies the local energy
inequality~\eqref{eq:generalized_energy_inequality}. Since $\phi\geq
0$, from~\eqref{eq:genapp} we have that $(\ue,\pe)$ satisfies
{
\begin{equation}
  \label{eq:gen1}
\begin{aligned}
  \int_{0}^{T}\int_\Omega |\nabla\ue|^{2}\phi\,dx
  ds&\leq\int_{0}^{T}\int_\Omega\Big[\frac{|\ue|^{2}}{2}(\phi_{t}+\Delta\phi)
  \\
  &+\big(\ue\frac{|\ue|^{2}}{2}
    +\pe\ue-\pe\nabla\pe-\ue\dive\ue \big)\cdot\nabla\phi\Big]\,dx ds.
\end{aligned}
\end{equation}}
By weak lower semicontinuity of the $L^{2}$-norm, the fact that
$\nabla\ue\rightharpoonup \nabla u$ weakly in 
$L^{2}((0,T)\times\Omega)$, and since $\phi\geq0$ we have that  
\begin{equation*}
  \int_{0}^{T}\int_\Omega|\nabla u|^{2}\,\phi\,dx ds\leq
  \liminf_{\e\rightarrow 0}\int_{0}^{T}\int_\Omega|\nabla
  \ue|^{2}\,\phi\,dx ds.
\end{equation*}
Since $\ue\rightarrow u$ strongly in $L^{2}((0,T)\times\Omega)$, we
get
\begin{equation*}
  \int_{0}^{T}\int_\Omega\frac{|\ue|^{2}}{2}\big(\phi_{t}+\Delta\phi\big)\,dx
  ds\rightarrow\int_{0}^{T}\int_\Omega\frac{|u|^{2}}{2}\big(\phi_{t}+\Delta\phi\big)\,dx
  ds,   \qquad\textrm{as $\e\rightarrow 0$}. 
\end{equation*}
Next, by interpolation we also have that $\ue\rightarrow u$ strongly
in $L^{2}(0,T;L^{3}(\Omega))$ and that $\ue$ is bounded in
$L^{4}(0,T;L^{3}(\Omega))$. Consequently, it also follows that
\begin{equation*}
  \int_{0}^{T}\int_\Omega
  \ue\frac{|\ue|^{2}}{2}\cdot\nabla\phi\,dx ds\rightarrow\int_{0}^{T}\int_\Omega
  u\frac{|u|^{2}}{2}\cdot\phi\,dx
  ds,\qquad\textrm{as $\e\rightarrow 0$}.
\end{equation*}
Now, we estimate the last two terms in~\eqref{eq:gen1}. We start by
estimating the term involving the pressure. We have that 
\begin{equation*}
  \pe\rightharpoonup p\textrm{ weakly in
  }L^{\frac{5}{3}}(\Omega\times(0,T)),
\end{equation*}
while by standard interpolation argument 
\begin{equation*}
\ue\rightarrow u\textrm{ strongly in }L^{\frac{5}{2}}(\Omega\times(0,T)).
\end{equation*}
These in turn imply that
\begin{equation*}
  \int_{0}^{T}\int_\Omega \pe\ue\cdot\nabla\phi\,dx
  ds\rightarrow\int_{0}^{T}\int_\Omega p\, u\cdot\nabla\phi\,dx ds,\qquad\textrm{as
    $\e\rightarrow 0$}. 
\end{equation*}
Concerning the integral 
\begin{equation*}
  A:=\e\int_{0}^{T}\int\pe\nabla\pe\cdot\nabla\phi\,dx ds, 
\end{equation*}
(quadratic in the pressure) we argue as follows: By using H\"older
inequality we get that
\begin{equation*}
\begin{aligned}
  |A|&\leq
  C\e\left(\int_{0}^T\|\nabla\pe\|_{\frac{5}{2}}^{\frac{5}{2}}\right)^{\frac{2}{5}}
\left(\int_0^T\|\pe\|_{\frac{5}{3}}^{\frac{5}{3}}\right)^{\frac{3}{5}}
  \\ 
  &\leq
  C\e\left(\int_{0}^T\|\nabla\pe\|_{\frac{5}{2}}^{\frac{5}{2}}\right)^{\frac{2}{5}},
\end{aligned}
\end{equation*}
and also  $A$ goes to $0$ as $\e\rightarrow 0$.

{
Finally, the we consider the last term, namely 
\begin{equation*}
B:=-\int_{0}^{T}\int_\Omega\ue\dive \ue\cdot\nabla\phi\,dx ds,
\end{equation*}
and since $\dive\ue$ converge weakly to $0$ in
$L^{2}(0,T;H^{-1}(\Omega))$ and $\nabla\ue$ is uniformly bounded in
$L^{2}(0,T;H^{1}(\Omega))$ we get, by uniqueness of the weak limits,
that
\begin{equation}
\dive\ue\rightharpoonup 0\quad\textrm{ weakly in }L^{2}((0,T)\times\Omega).
\end{equation}
By using the fact the $\ue\to u$ strongly in $L^{2}((0,T)\times\Omega)$
we have that $B$ goes to $0$ when $\e$ vanishes}, hence
that~\eqref{eq:generalized_energy_inequality} is satisfied. 

\begin{thebibliography}{10}

\bibitem{AR2014} 
  C.~ Amrouche, and A.~Rejaiba, \textit{$L^p$-theory for {S}tokes and
{N}avier-{S}tokes equations with {N}avier boundary condition},
{J. Differential Equations}, \textbf{256} {(2014)}, no.~{4},
{1515--1547},


\bibitem{Bei1985a}
H.~{Beir{\~a}o da Veiga}, \emph{On the suitable weak solutions to the {N}avier-{S}tokes
  equations in the whole space}, J. Math. Pures Appl. (9) \textbf{64} (1985),
  no.~1, 77--86. 

\bibitem{Bei1985b}
H.~{Beir{\~a}o da Veiga}, \emph{On the construction of suitable weak solutions
  to the {N}avier-{S}tokes equations via a general approximation theorem}, J.
  Math. Pures Appl. (9) \textbf{64} (1985), no.~3, 321--334. 


\bibitem{Bei1986}
H.~{Beir{\~a}o da Veiga}, \emph{Local energy inequality and singular set of weak solutions of
  the boundary non-homogeneous navier-stokes problem}, Current topics in
  partial differential equations (Y.~Ohya, K.~Kasahara, and N.~Shimakura,
  eds.), Kinokuniya Company Ltd., Tokyo, 1986, Papers dedicated to Professor
  Sigeru Mizohata on the occasion of his sixtieth birthday.

\bibitem{Bei2004}
H.~{Beir{\~a}o da Veiga}, \emph{Regularity for {S}tokes and generalized {S}tokes systems under
  nonhomogeneous slip-type boundary conditions}, Adv. Differential Equations
  \textbf{9} (2004), no.~9-10, 1079--1114. 

\bibitem{Bei2006}
H.~{Beir{\~a}o da Veiga}, \emph{Vorticity and regularity for flows under the {N}avier boundary
  condition}, Commun. Pure Appl. Anal. \textbf{5} (2006), no.~4, 907--918.

\bibitem{Bei2007c}
H.~{Beir{\~a}o da Veiga}, \emph{Remarks on the {N}avier-{S}tokes evolution equations under slip
  type boundary conditions with linear friction}, Port. Math. (N.S.)
  \textbf{64} (2007), no.~4, 377--387. 

\bibitem{Ber2010b}
L.~C. Berselli, \emph{Some results on the {N}avier-{S}tokes equations with
  {N}avier boundary conditions}, Riv. Math. Univ. Parma (N.S.) \textbf{1}
  (2010), no.~1, 1--75, Lecture notes of a course given at SISSA/ISAS, Trieste,
  Sep. 2009.

\bibitem{BCI2007}
A.~Biryuk, W.~Craig, and S.~Ibrahim, \emph{Construction of suitable weak
  solutions of the {N}avier-{S}tokes equations}, Stochastic analysis and
  partial differential equations, Contemp. Math., vol. 429, Amer. Math. Soc.,
  Providence, RI, 2007, pp.~1--18.

\bibitem{BMR2007}
M.~Bul{\'{\i}}{\v{c}}ek, J.~M{\'a}lek, and K.~R. Rajagopal, \emph{Navier's slip
  and evolutionary {N}avier-{S}tokes-like systems with pressure and shear-rate
  dependent viscosity}, Indiana Univ. Math. J. \textbf{56} (2007), no.~1,
  51--85.

\bibitem{CKN1982}
L.~Caffarelli, R.~Kohn, and L.~Nirenberg, \emph{Partial regularity of suitable
  weak solutions of the {N}avier-{S}tokes equations}, Comm. Pure Appl. Math.
  \textbf{35} (1982), no.~6, 771--831.

\bibitem{CL2014}
T.~Chac{\'o}n~Rebollo and R.~Lewandowski, \emph{Mathematical and numerical
  foundations of turbulence models and applications}, Modeling and Simulation
  in Science, Engineering and Technology, Birkh\"auser/Springer, New York,
  2014.

\bibitem{Cho1968}
A.~J. Chorin, \emph{Numerical solution of the {N}avier-{S}tokes equations},
  Math. Comp. \textbf{22} (1968), 745--762.

\bibitem{Cho1969}
A.~J. Chorin, \emph{On the convergence of discrete approximations to the
  {N}avier-{S}tokes equations}, Math. Comp. \textbf{23} (1969), 341--353.

\bibitem{DM2006}
D.~Donatelli and P.~Marcati, \emph{A dispersive approach to the artificial
  compressibility approximations of the {N}avier-{S}tokes equations in 3{D}},
  J. Hyperbolic Differ. Equ. \textbf{3} (2006), no.~3, 575--588. 

\bibitem{DS2011}
D.~Donatelli and S.~Spirito, \emph{Weak solutions of {N}avier-{S}tokes
  equations constructed by artificial compressibility method are suitable}, J.
  Hyperbolic Differ. Equ. \textbf{8} (2011), no.~1, 101--113.


\bibitem{FP2000b}
E.~Feireisl and H.~Petzeltov{\'a}, \emph{On integrability up to the boundary of
  the weak solutions of the {N}avier-{S}tokes equations of compressible flow},
  Comm. Partial Differential Equations \textbf{25} (2000), no.~3-4, 755--767.

\bibitem{Gal2011}
G.~P. Galdi, \emph{An introduction to the mathematical theory of the
  {N}avier-{S}tokes equations. {S}teady-state problems.}, Springer Monographs
  in Mathematics, Springer-Verlag, New York, 2011.

\bibitem{GT1998}
D.~Gilbarg and N.~S. Trudinger, \emph{Elliptic partial differential equations
  of second order}, Classics in Mathematics, Springer-Verlag, Berlin, 2001,
  Reprint of the 1998 edition.

\bibitem{Gue2006}
J.-L. Guermond, \emph{Finite-element-based {F}aedo-{G}alerkin weak solutions to
  the {N}avier-{S}tokes equations in the three-dimensional torus are suitable},
  J. Math. Pures Appl. (9) \textbf{85} (2006), no.~3, 451--464.

\bibitem{Gue2007}
J.-L. Guermond, \emph{Faedo-{G}alerkin weak solutions of the {N}avier-{S}tokes
  equations with {D}irichlet boundary conditions are suitable}, J. Math. Pures
  Appl. (9) \textbf{88} (2007), no.~1, 87--106.

\bibitem{Gue2008}
J.-L. Guermond, \emph{On the use of the notion of suitable weak solutions in {CFD}},
  Internat. J. Numer. Methods Fluids \textbf{57} (2008), no.~9, 1153--1170.

\bibitem{Hop1951}
E.~Hopf, \emph{\"{U}ber die {A}nfangswertaufgabe f\"ur die hydrodynamischen
  {G}rundgleichungen}, Math. Nachr. \textbf{4} (1951), 213--231. 

\bibitem{IP2006}
D.~Iftimie and G.~Planas, \emph{Inviscid limits for the {N}avier-{S}tokes
  equations with {N}avier friction boundary conditions}, Nonlinearity
  \textbf{19} (2006), no.~4, 899--918.

\bibitem{LS1999}
O.~A. Lady{\v{z}}enskaya and G.~A. Seregin, \emph{On partial regularity of
  suitable weak solutions to the three-dimensional {N}avier-{S}tokes
  equations}, J. Math. Fluid Mech. \textbf{1} (1999), no.~4, 356--387.

\bibitem{Ler1934}
J.~Leray, \emph{Sur le mouvement d'un liquide visqueux emplissant l'espace},
  Acta Math. \textbf{63} (1934), no.~1, 193--248. 

\bibitem{Lio1998}
P.-L. Lions, \emph{Mathematical topics in fluid mechanics. {V}ol. 2}, Oxford
  Lecture Series in Mathematics and its Applications, vol.~10, The Clarendon
  Press, Oxford University Press, New York, 1998, Compressible models, Oxford
  Science Publications.

\bibitem{Rus2012}
W.~Rusin, \emph{Incompressible 3{D} {N}avier-{S}tokes equations as a limit of a
  nonlinear parabolic system}, J. Math. Fluid Mech. \textbf{14} (2012), no.~2,
  383--405. 

\bibitem{SvW1986}
H.~Sohr and W.~von Wahl, \emph{On the regularity of the pressure of weak
  solutions of {N}avier-{S}tokes equations}, Arch. Math. (Basel) \textbf{46}
  (1986), no.~5, 428--439. 

\bibitem{SS1973}
V.~A. Solonnikov and V.~E. {\v{S}}{\v{c}}adilov, \emph{A certain boundary value
  problem for the stationary system of {N}avier-{S}tokes equations}, Trudy Mat.
  Inst. Steklov. \textbf{125} (1973), 196--210, 235, Boundary value problems of
  mathematical physics, 8.

\bibitem{Tem1968}
R.~Temam, \emph{Une m\'ethode d'approximation de la solution des \'equations de
  {N}avier-{S}tokes}, Bull. Soc. Math. France \textbf{96} (1968), 115--152.

\bibitem{Tem1969}
R.~Temam, \emph{Sur l'approximation de la solution des \'equations de
  {N}avier-{S}tokes par la m\'ethode des pas fractionnaires. {I}}, Arch.
  Rational Mech. Anal. \textbf{32} (1969), 135--153. 

\bibitem{XX2007}
Y.~Xiao and Z.~Xin, \emph{On the vanishing viscosity limit for the 3{D}
  {N}avier-{S}tokes equations with a slip boundary condition}, Comm. Pure Appl.
  Math. \textbf{60} (2007), no.~7, 1027--1055. 

\end{thebibliography}
\def\cprime{$'$} \def\ocirc#1{\ifmmode\setbox0=\hbox{$#1$}\dimen0=\ht0
  \advance\dimen0 by1pt\rlap{\hbox to\wd0{\hss\raise\dimen0
  \hbox{\hskip.2em$\scriptscriptstyle\circ$}\hss}}#1\else {\accent"17 #1}\fi}
  \def\cprime{$'$} \def\polhk#1{\setbox0=\hbox{#1}{\ooalign{\hidewidth
  \lower1.5ex\hbox{`}\hidewidth\crcr\unhbox0}}}
\providecommand{\bysame}{\leavevmode\hbox to3em{\hrulefill}\thinspace}
\providecommand{\MR}{\relax\ifhmode\unskip\space\fi MR }
\providecommand{\MRhref}[2]{%
  \href{http://www.ams.org/mathscinet-getitem?mr=#1}{#2}
}
\providecommand{\href}[2]{#2}

\end{document}